\documentclass[11pt,a4paper,reqno,oneside]{amsart}

\author{Nicholas Fleming-Vázquez}
 \address{Department of Mathematics, University of Toronto, 40
 	St George St., Toronto, ON M5S 2E4, Canada } 
  \urladdr{\url{https://nflemingvazquez.github.io/}}
 \thanks{This work is based on material from the author’s PhD thesis; he gratefully acknowledges funding from the University of Warwick and helpful feedback from Ian Melbourne and the referee.}
\title[Rates in the multivariate WIP for nonuniformly hyperbolic maps]{Rates of convergence in the multivariate weak invariance principle for nonuniformly hyperbolic maps}
\usepackage{amssymb,amsmath,amsthm}
\usepackage{mathtools}
\usepackage{mathrsfs}
\usepackage{graphicx}
\usepackage{xtab}
\usepackage{xcolor}
\usepackage{bbm}
\usepackage{enumitem}
\usepackage{url}
\usepackage{bbm}
\usepackage{tikz}
\usetikzlibrary{matrix,arrows}

 \newcommand{\Z}{\ensuremath{\mathbb{Z}}}
 \newcommand{\R}{\ensuremath{\mathbb{R}}}

\setlist[enumerate,1]{label=(\roman*)}

\newcommand{\norm}[1]{\left\lVert#1\right\rVert}
\newcommand{\pnorm}[1]{\left|#1\right|}
\newcommand{\norminf}[1]{\left|#1\right|_\infty}
\newcommand{\eqd}{=_d}
\newcommand{\holsp}[1][\eta]{\mathcal{C}^{#1}}

\newcommand{\Ex}[2][]{\mathbb{E}_{#1\!\!}\left[#2\right]}

\newcommand{\tf}{T}
\newcommand{\ms}{M}

\newcommand{\eps}{\varepsilon}

\newcommand{\map}[2]{\colon#1\rightarrow#2}
\newcommand{\Lip}{\mathrm{Lip}}

\makeatletter
\let\save@mathaccent\mathaccent
\newcommand*\if@single[3]{%
	\setbox0\hbox{${\mathaccent"0362{#1}}^H$}%
	\setbox2\hbox{${\mathaccent"0362{\kern0pt#1}}^H$}%
	\ifdim\ht0=\ht2 #3\else #2\fi
}
\newcommand*\rel@kern[1]{\kern#1\dimexpr\macc@kerna}
\newcommand*\widebar[1]{\@ifnextchar^{{\wide@bar{#1}{0}}}{\wide@bar{#1}{1}}}
\newcommand*\wide@bar[2]{\if@single{#1}{\wide@bar@{#1}{#2}{1}}{\wide@bar@{#1}{#2}{2}}}
\newcommand*\wide@bar@[3]{%
	\begingroup
	\def\mathaccent##1##2{%
		\let\mathaccent\save@mathaccent
		\if#32 \let\macc@nucleus\first@char \fi
		\setbox\z@\hbox{$\macc@style{\macc@nucleus}_{}$}%
		\setbox\tw@\hbox{$\macc@style{\macc@nucleus}{}_{}$}%
		\dimen@\wd\tw@
		\advance\dimen@-\wd\z@
		\divide\dimen@ 3
		\@tempdima\wd\tw@
		\advance\@tempdima-\scriptspace
		\divide\@tempdima 10
		\advance\dimen@-\@tempdima
		\ifdim\dimen@>\z@ \dimen@0pt\fi
		\rel@kern{0.6}\kern-\dimen@
		\if#31
		\overline{\rel@kern{-0.6}\kern\dimen@\macc@nucleus\rel@kern{0.4}\kern\dimen@}%
		\advance\dimen@0.4\dimexpr\macc@kerna
		\let\final@kern#2%
		\ifdim\dimen@<\z@ \let\final@kern1\fi
		\if\final@kern1 \kern-\dimen@\fi
		\else
		\overline{\rel@kern{-0.6}\kern\dimen@#1}%
		\fi
	}%
	\macc@depth\@ne
	\let\math@bgroup\@empty \let\math@egroup\macc@set@skewchar
	\mathsurround\z@ \frozen@everymath{\mathgroup\macc@group\relax}%
	\macc@set@skewchar\relax
	\let\mathaccentV\macc@nested@a
	\if#31
	\macc@nested@a\relax111{#1}%
	\else
	\def\gobble@till@marker##1\endmarker{}%
	\futurelet\first@char\gobble@till@marker#1\endmarker
	\ifcat\noexpand\first@char A\else
	\def\first@char{}%
	\fi
	\macc@nested@a\relax111{\first@char}%
	\fi
	\endgroup
}
\makeatother

\usepackage{hyperref}
\usepackage[a4paper,margin=1.25in]{geometry}
\hypersetup{colorlinks=true,
	linkcolor=blue,
	filecolor=magenta,      
	urlcolor=cyan,
	pdftitle={Rates of convergence in the multivariate weak invariance principle for nonuniformly hyperbolic maps},%
	pdfauthor={Nicholas Fleming-Vázquez}}

\newcommand{\wass}{\mathcal{W}}
\newcommand{\kant}{\wass_1}

\newcommand{\shat}[1]{\vphantom{#1}\smash[t]{\widehat{#1}}}
\newcommand{\stilde}[1]{\vphantom{#1}\smash[t]{\widetilde{#1}}}

\newcommand{\sepdholsp}[1]{\mathcal{H}^\eta_{#1}}
\newcommand{\dsemi}[1]{[#1]_{\eta}}
\newcommand{\sdsemi}[2]{[#1]_{\eta,#2}}
\newcommand{\dhnorm}[1]{\norm{#1}_\eta}

\newtheorem{theorem}{Theorem}[section]
\newtheorem{lemma}[theorem]{Lemma}
\newtheorem{prop}[theorem]{Proposition}
\newtheorem{corol}[theorem]{Corollary}
\newtheorem*{notation}{Notation}

\theoremstyle{definition}
\newtheorem{defn}[theorem]{Definition}
\newtheorem{example}[theorem]{Example}

\theoremstyle{remark}
\newtheorem{remark}[theorem]{Remark}

\numberwithin{equation}{section}

\begin{document}


\date{}

\dedicatory{}

	\begin{abstract}
		We obtain rates of convergence in the weak invariance principle for $\R^d$-valued H\"older observables over nonuniformly hyperbolic maps. In particular, for maps modelled by a Young tower with superpolynomial tails (e.g.\ the Sinai billiard map, and Axiom A diffeomorphisms) we obtain a rate of $O(n^{-r})$ in the Wasserstein $p$-metric for all $r<1/4$ and $p<\infty$. Additionally, this is the first result on rates that covers certain invertible, slowly mixing maps, such as Bunimovich flowers.
	\end{abstract}
    \maketitle
		\section{Introduction}
Let $T:M\to M$ be a chaotic dynamical system with ergodic invariant probability measure $\mu$. Let $v:M\to \R^d$ be an observable with $\int v\, d\mu = 0 $ and let $S_n = \sum_{i=0}^{n-1}{v\circ T^i}$ be its sequence of Birkhoff sums. Consider the c\`adl\`ag process $W_n \in D([0,1],\R^d)$ defined by $W_n(t) = n^{-1/2} S_{[nt]}$. Then $v$ is said to satisfy the \textit{weak invariance principle (WIP)} if $W_n$ weakly converges to a Brownian motion. 

For many classes of dynamical systems the WIP is satisfied by all sufficiently regular observables. For example, for nonuniformly hyperbolic/expanding maps modelled by a Young tower whose return time function is square-integrable~\cite{young1998statistical,young1999recurrence}, the WIP is satisfied by all H\"older observables. This class includes many examples such as Axiom A diffeomorphisms, dispersing billiards and certain intermittent maps.

The problem of quantifying the convergence rate in the WIP for dynamical systems has attracted much interest recently~\cite{antoniou_melbourne,dedecker_merlevede_rio,liu_wang_stationary,liu_wang_sequential,paviato_rates_article,leppanen_nakajima_nakano,liu_saussol_vaienti_wang,melbourne_wang}. 

Antoniou \& Melbourne~\cite{antoniou_melbourne} obtained rates in the Prokhorov metric for $\R$-valued H\"older observables of nonuniformly expanding maps. In the same setting, Liu \& Wang \cite{liu_wang_stationary} obtained rates in the Wasserstein-$p$ metric and Dedecker, Merlev\`ede \& Rio~\cite{dedecker_merlevede_rio} established improved rates in the Wasserstein-$2$ metric. In particular, for systems modelled by Young towers with superpolynomial tails, \cite{antoniou_melbourne} obtained a rate of $O(n^{-\kappa})$ in the Prokhorov metric for all $\kappa<\frac{1}{4}$, \cite{liu_wang_stationary} obtained the same rate in the Wasserstein-$p$ metric for all $p<\infty$, and \cite{dedecker_merlevede_rio} achieved a sharper bound of $O(n^{-1/4}(\log n)^{1/4})$ in the Wasserstein-2 metric. The papers~\cite{antoniou_melbourne,liu_wang_stationary} rely on the martingale approximation method, whereas~\cite{dedecker_merlevede_rio} proves strong approximation results for a class of weakly dependent sequences of random variables. The method in~\cite{dedecker_merlevede_rio} is based on constructing an explicit approximating sequence of iid Gaussian random variables and controlling the approximation error via a quantitative conditional central limit theorem together with a Fuk–Nagaev type inequality.

The only existing result on rates in the \emph{multivariate} WIP (that is for $\R^d$-valued observables with $d>1$) for dynamical systems is due to Paviato~\cite{paviato_rates_article}, and applies to the aforementioned class of nonuniformly expanding maps, and also to nonuniformly hyperbolic flows. He obtained rates in the Wasserstein-$1$ metric that are independent of $d$ but are
at best $O(n^{-1/6})$ for $d > 1$. This is because he uses martingale methods, and obtaining better rates in the multivariate WIP remains an open problem for martingales (see the introduction of~\cite{multidim_martingale_asip} for more details).

In this paper, we obtain rates in the multivariate WIP in the Wasserstein-$p$ metric for nonuniformly hyperbolic maps modelled by Young towers. In particular, for systems modelled by a Young tower with superpolynomial tails we obtain a rate of $O(n^{-\kappa})$ in the Wasserstein-$p$ metric for all $\kappa<\frac{1}{4}$ and $p<\infty$. In contrast to existing results on rates in the WIP for nonuniformly hyperbolic maps, we do not assume exponential contraction along stable leaves so our results apply to slowly mixing invertible systems such as Bunimovich flowers~\cite{bunimovich_flowers} (see Example~\ref{exa:flowers}).

In our main theorem (Theorem~\ref{thm:rates_wip_lip}) we obtain rates in the WIP under a weak dependence condition called the functional correlation bound, where the rate obtained depends on the polynomial rate of decay in the functional correlation bound. By~\cite{fleming2022}, this condition is satisfied by systems modelled by Young towers with $O(n^{-\beta})$ tails for $\beta>2$. The proof of our main theorem is based on Bernstein's classical `big block-small block' method and is loosely inspired by the proof of the central limit theorem in~\cite[Section 7.8]{chernov_markarian}, see Remark~\ref{remark:chernov_clt} for more details.

We end this introduction by illustrating the rates we obtain in the WIP for several examples.
\begin{example}\label{exa:lsv_and_baker}
	Let $\alpha\in (0,1)$. Define $g:[0,1/2)\to [0,1]$ by $g(x) = x(1+2^\alpha x^\alpha)$. Consider the intermittent map $T:[0,1]\to [0,1]$ defined by 
	\[ 	T(x) = \begin{cases}
		g(x), & x \in [0,1/2),\\
		2x - 1, & x \in [1/2,1].
	\end{cases} \]
The LSV map $T$ is a prototypical example of a nonuniformly expanding map~\cite{lsvmaps} and has a unique absolutely continuous invariant probability measure. As in~\cite{melbourne_varandas_2016note}, consider an intermittent baker's map $S:[0,1]^2 \to [0,1]^2$ defined by 
\begin{equation*}
	S(x,y) = \begin{cases}
		(T(x),g^{-1}(y)), & x \in [0,1/2),\\
		(T(x),(y+1)/2), & x \in [1/2,1].
	\end{cases}
\end{equation*}
The map $S$ is invertible and has a unique physical measure. Let $\alpha\in (0,1/2)$. Then H\"older observables satisfy the WIP for both maps $T$ and $S$. Prior results on rates in the WIP cover the map $T$ but they do not apply to the nonuniformly hyperbolic map $S$ since $S$ contracts  at a polynomial rate along stable manifolds (which consist of vertical line segments). 

Set $\gamma = 1/\alpha-1$. We prove rates in the Wasserstein-1 metric for $\R^d$-valued H\"older observables for both maps $T$ and $S$. The rates are of the form $O(n^{-\kappa(\gamma)})$ where 
\[ \kappa(\gamma)= \frac{(\gamma-1)(2\gamma-1)}{2(4\gamma^2+\gamma-1)}. \]
More generally, we obtain rates in the Wasserstein-$p$ metric for all $1\le p< 2\gamma$. 

Let us now compare the rates we obtain with existing results for the map $T$. For $d=1$,~\cite{antoniou_melbourne} obtained a rate of $O(n^{-1/4+\alpha/2+\eps})$ in the Prokhorov metric, whilst in the restricted range $\alpha\in (0,1/4)$~\cite{liu_wang_stationary} obtained a rate of $O(n^{-1/4+\frac{\alpha}{4(1-\alpha)}+\eps})$ in the Wasserstein-$p$ metric for any $\eps>0$ and $p<\frac{1}{2\alpha}$. For $p=2$,~\cite{dedecker_merlevede_rio} obtained an improved rate of the form
\[ \begin{cases} 
	O(n^{\alpha-\frac12}(\log n)^{\frac{1}{2}}), & \alpha \in (\frac{1}{4},\frac{1}{2})\\
	O(n^{-\frac14}(\log n)^{\frac{1}{4}}), & \alpha \in (0,\frac{1}{4}].
\end{cases} \]
For all $d\ge 1$,~\cite{paviato_rates_article} obtained rates in the Wasserstein-$1$ metric of the form 
\[ \begin{cases}
	O(n^{\alpha - \frac12 + \eps}), & \alpha \in [\frac{1}{3},\frac{1}{2}) \\
	O(n^{-\frac{1}{6}+\eps}), & \alpha \in (0,\frac{1}{3}),
\end{cases}\]
where $\eps>0$ is arbitrary.

In particular, for the map $T$ our rates improve on those in~\cite{paviato_rates_article} for $d>1$ and $\alpha \le \frac{1}{16}(7-\sqrt{17})$. However, for $d=1$ the rates obtained in~\cite{dedecker_merlevede_rio} are better than ours for any $\alpha\in (0,1/2)$.
\end{example}
\begin{example}\label{exa:eslami_liverani}
    In~\cite{eslami_liverani}, Eslami \& Liverani considered a class of area-preserving $C^\infty$ almost Anosov maps on $\R^2/\Z^2$ with a neutral fixed point at $(0,0)$. The maps are of the form 
$T(x,y) = (x+h(x)+y,h(x)+y)$, where $h(x) = bx^3 + O(x^5)$ for some $b>0$. We obtain a rate of $O(n^{-5/38})$ in the $1$-Wasserstein metric for $\R^d$-valued H\"older observables.
\end{example}
\begin{example}[Bunimovich flowers]\label{exa:flowers}
	In~\cite{bunimovich_flowers}, Bunimovich introduced a class of planar chaotic billiards with both dispersing and focusing boundary components, which are now known as Bunimovich flowers. Let $T$ be the billiard map and let $\mu$ be the unique absolutely continuous invariant probability measure. We obtain a rate of $O(n^{-3/34})$ in the 1-Wasserstein metric for $\R^d$-valued H\"older observables.
\end{example}
See~\cite[Section 5]{balint_komalovics} for a detailed description of Bunimovich flowers and other examples of nonuniformly hyperbolic maps to which our results apply.
	\begin{notation}
		We write $a_n=O(b_n) $ or $ a_n\ll b_n $ if there exists a constant $ C>0 $ such that $ a_n\le Cb_n $ for all $ n\ge 1. $ We write $ a_n\sim b_n $ if $ \lim_n a_n/b_n=1. $
		
		For $ d\ge 1 $ and $ a,b\in \R^d $ denote $ a\otimes b=ab^T. $ We endow $ \R^d $ with the norm $ |y|=\sum_{i=1}^d |y_i|. $ Let $ \eta\in (0,1] $. Recall that $ v\map{\ms}{\R^d} $ is an $ \eta $-H\"older observable on a metric space $(M,\rho)$ if $ \dhnorm{v}=\norminf{v}+\dsemi{v} <\infty$, where $\dsemi{v}=\sup_{x\ne y}\frac{|v(x)-v(y)|}{\rho(x,y)^\eta}$. If $ \eta=1 $ we call $ v $ Lipschitz and write $ \Lip(v)=[v]_1 $.
	\end{notation}	
	The rest of this article is structured as follows. In Section~\ref{sec:main_result}, we recall the definition of the Wasserstein-$p$ metric, introduce the functional correlation bound that we consider and state our main result. In Section~\ref{sec:prelims}, we recall some results from the probability theory literature, as well as some consequences of the functional correlation bound. Finally, in Section~\ref{sec:proof_main_result} we prove our main result.
	\section{Setup and main result}\label{sec:main_result}
	\subsection{The Wasserstein $p$-metric} We consider the space $D = D([0,1],\R^d)$ of c\`adl\`ag functions $x:[0,1]\to \R^d$ with the supremum norm $\norminf{x}=\sup_{0\le t\le 1}|x(t)|$. 

	Given Radon probability measures $\mu$ and $\nu$ on $D$, we let $\Pi(\mu,\nu)$ denote the set of Radon probability measures on $D\times D$ for which the projections on the first and second factors coincide with $ \mu $ and $ \nu $, respectively.

	For $ p\ge 1 $, let $ \mathcal{P}_r^p(D) $ denote the set of Radon probability measures on $ D $ such that $ x\mapsto \norminf{x}^p $ is integrable. The \textit{Wasserstein-$ p $ metric} on $ \mathcal{P}_r^p(D) $ is defined by 
	\[ \wass_p(\mu,\nu)=\inf_{\sigma\in \Pi(\mu,\nu)}\left(\int_{D\times D}\norminf{x-y}^pd\sigma(x,y)\right)^{1/p}. \]
	If $X$ and $Y$ are random elements of $D$ with respective laws $\mu$ and $\nu$, we write $\wass_p(X,Y)=\wass_p(\mu,\nu)$.
	
	By the Kantorovich-Rubinstein Theorem, for all $\mu$, $\nu\in \mathcal{P}_r^1(D)$, we have that
	\begin{equation}\label{eq:kantorovich_rubenstein}
		\kant(\mu,\nu) = \sup\bigg\{\int f\, d\mu-\int f\, d\nu \bigg| f:D\to \R\text{ is Lipschitz and } \Lip(f)\le 1 \bigg\}.
	\end{equation}
	See~\cite[Chapter 3]{bogachev_weak_conv} for more details on the Wasserstein-$p$ metric.
	\subsection{The Functional Correlation Bound}
	Let $ (M,\rho) $ be a metric space and fix $ \eta\in(0,1] $. Recall that $\dsemi{v}=\sup_{x\ne y}\frac{|v(x)-v(y)|}{\rho(x,y)^\eta}$ denotes the $\eta$-H\"older seminorm of an observable $v:M\to \R$. Let $ q\ge 1$ be an integer. Given a function $ G\map{\ms^{q}}{\R} $ and $ 0\le i< q$ we denote
		\[\sdsemi{G}{i} =\sup_{x_0,\dots,x_{q-1}\in \ms}\dsemi{G(x_0,\dots,x_{i-1},\cdot,x_{i+1},\dots,x_{q-1})}. \]
		We call $ G $ \textit{separately $ \eta $-H\"older}, and write $ G\in \sepdholsp{q}(\ms)$, if $ \norminf{G}+\sum_{i=0}^{q-1} \sdsemi{G}{i}{<\infty}. $

	Let $\gamma>0$. We consider dynamical systems which satisfy the following property:
	\begin{defn}\label{def:fcb}
		Let $ \mu $ be a Borel probability measure and let $ T:M\to M $ be an ergodic $ \mu $-preserving transformation. Suppose that there exists a constant $C>0$ such that for all integers $ 0\le p< q,\ 0\le n_0\le \cdots\le n_{q-1} $,
		\begin{align}
			&\bigg|\int_\ms  G(T^{n_0}x,\dots,T^{n_{q-1}}x)d\mu(x)\nonumber\\
			&\qquad\qquad-\int_{\ms^2} G(T^{n_0}x_0,\dots,T^{n_{p-1}}x_0,T^{n_p}x_1,\dots,T^{n_{q-1}} x_1)d\mu(x_0)d\mu(x_1)\bigg|	\nonumber\\
			&\quad\le C(n_p-n_{p-1})^{-\gamma}\,\biggl(\norminf{G}+\sum_{i=0}^{q-1} \sdsemi{G}{i}\biggr)	\label{eq:fcb_bd}	
		\end{align}
		for all $G\in \sepdholsp{q}(\ms)$. Then we say that $T$ satisfies the Functional Correlation Bound with rate $n^{-\gamma}$.
	\end{defn}
	\begin{remark}\label{rk:fcb_implies_doc}If $T:M\to M$ satisfies the Functional Correlation Bound with rate $n^{-\gamma}$, then decay of correlations holds for $\eta$-H\"older observables with the same rate. Indeed, let $v,w\in \holsp(M,\R)$. Then setting $G(x_0,x_1)=v(x_0)w(x_1)$, we obtain that for all $n\ge 1$,
		\[\left|\int_M v\, w\circ T^n d\mu - \int_M v\, d\mu \int_M w\, d\mu \right|\le C\dhnorm{v}\dhnorm{w}n^{-\gamma}.\]
	\end{remark}	
	By~\cite[Theorem 2.3]{fleming2022}, if $T:M\to M$ is a mixing nonuniformly hyperbolic map modelled by a Young tower with tails of the form $O(n^{-(\gamma+1)})$, then $T$ satisfies the Functional Correlation Bound with rate $n^{-\gamma}$.

	The Functional Correlation Bound is inspired by a condition of the same name introduced by Lepp\"anen~\cite{leppanen2017functional}. However, for measure-preserving transformations the condition considered by Lepp\"anen is slightly more restrictive than Definition~\ref{def:fcb}; in particular, the term $\sum_{i=0}^{q-1}\sdsemi{G}{i}$ in~\eqref{eq:fcb_bd} is replaced by $\max_{0\le i<q}\sdsemi{G}{i}$.
	\subsection{Rates in the WIP} 
	Let $ T:M\to M $ satisfy the Functional Correlation Bound with rate $ n^{-\gamma} $, $ \gamma>1 $. Let $v\in \holsp(M,\R^d)$ be mean zero. For $n\ge 1$ and $t\in [0,1]$, we define
	$ S_n = \sum_{i=0}^{n-1}v\circ T^i $ and $W_n(t) = n^{-1/2}S_{[nt]}$. We view $W_n$ as a random element of $D([0,1],\R^d)$ defined on the probability space $(M,\mu)$. Note that the law of $W_n$ is a Radon measure since we can write $W_n = \pi(v,\dots,v\circ T^{n-1})$ where $\pi:(\R^d)^n \to D([0,1],\R^d)$ is continuous.

	 By~\cite[Theorem 1.2]{fleming_fcb_homogenisation}, the observable $v$ satisfies the WIP:
	\begin{lemma}
		Let $ T:M\to M $ satisfy the Functional Correlation Bound with rate $ n^{-\gamma} $, $ \gamma>1 $, and suppose that $ v\in \holsp(M,\R^d)$ with mean zero. Then the limit $\Sigma = \lim_{n\to \infty}n^{-1}\Ex[\mu]{S_n \otimes S_n}$ exists. Moreover, $W_n$ converges weakly in $D([0,1],\R^d)$ to a Brownian motion $W$ with covariance $\Sigma$.
	\end{lemma}
	Our main theorem gives a bound on the rate of convergence in the WIP in the Wasserstein-$r$ metric for $1\le r<2\gamma$:
\begin{theorem}\label{thm:rates_wip_lip}
	Let $ T:M\to M $ satisfy the Functional Correlation Bound with rate $ n^{-\gamma} $, $ \gamma>1 $, and suppose that $ v\in \holsp(M,\R^d)$ with mean zero. Let $1\le r<2\gamma$. Then there is a constant $ C>0 $ such that $ \wass_r(W_n,W)\le Cn^{-\kappa(\gamma,r)} $ for all $ n\ge 1 $, where 
	\begin{equation*}
		\kappa(\gamma,r)=\frac{(2\gamma-r)(2\gamma-1)(\gamma-1)}{2(8\gamma^3 - r(2\gamma^2 + 3\gamma - 1))}.
	\end{equation*}
\end{theorem}
\begin{remark}
	(i) This result yields the bound $\pi(W_n,W)=O\big(n^{-\frac{r}{r+1}\kappa(\gamma,r)}\big)$ in the Prokhorov metric. Indeed, $\pi(\mu,\nu)\le \wass_r(\mu,\nu)^{\frac{r}{r+1}}$ for any $r\in [1,\infty)$ and any probability measures $\mu$, $\nu$ on $D=D([0,1],\R^d)$.

    (ii) This result also gives a convergence rate in the central limit theorem in the Wasserstein-$r$ metric: by considering the projection $x\in D\mapsto x(1)$, we obtain \[\wass_r(n^{-\frac12}S_n,W(1))\le \wass_r(W_n,W).\]

    (iii) For probability measures $\mu$, $\nu$ on $\R$ such that $\nu$ is absolutely continuous with density $g$, the Kolmogorov metric $d_K$ satisfies $d_K(\mu,\nu)\le (1+\norminf{g})\pi(\mu,\nu)$ (see~\cite{gibbs_su}). Hence for $d=1$, we obtain the (suboptimal) Berry-Esseen type bound $d_K(n^{-\frac12}S_n,W(1))=O\big(n^{-\frac{r}{r+1}\kappa(\gamma,r)}\big)$.
\end{remark}
Note that in particular $\lim_{\gamma\to\infty}\kappa(\gamma,r)=\frac{1}{4}$ for all $r\in (1,\infty)$. The exact formula for $ \kappa$ could possibly be improved by more careful arguments. However, since our proof is based on Bernstein's `big block-small block' technique, it does not seem possible to obtain a rate better than $ O(n^{-1/4}) $.

As mentioned above, if $T:M\to M$ is a mixing nonuniformly hyperbolic map modelled by a Young tower with $O(n^{-(\gamma+1)})$ tails, then $T$ satisfies the Functional Correlation Bound with rate $n^{-\gamma}$ by~\cite[Theorem 2.3]{fleming2022}. Combining this with Theorem~\ref{thm:rates_wip_lip} yields the following corollary:
\begin{corol}\label{corol:wip_rates_for_nuh}
    Let $T:M\to M$ be a mixing transformation modelled by a Young tower with $O(n^{-(\gamma+1)})$ tails, where $\gamma>1$. Let $v\in\holsp(M,\R^d) $ with mean zero and let $1\le r<2\gamma$. Then there is a constant $ C>0 $ such that $ \wass_r(W_n,W)\le Cn^{-\kappa(\gamma,r)} $ for all $ n\ge 1 $.
\end{corol}
The rates given in Examples~\ref{exa:lsv_and_baker}-\ref{exa:flowers} follow immediately from Corollary~\ref{corol:wip_rates_for_nuh}. Indeed, the LSV map $T$ and intermittent baker's map $S$ from Example~\ref{exa:lsv_and_baker} are modelled by Young towers with $O(n^{-1/\alpha})$ tails~\cite{young1999recurrence,melbourne_varandas_2016note}.
By~\cite{eslami_liverani}, Example~\ref{exa:eslami_liverani} is modelled by a Young tower with $O(n^{-4})$ tails.
Finally, the rate given in Example~\ref{exa:flowers} for Bunimovich flowers follows from  the fact that the billiard map is modelled by a Young tower with $O(n^{-3})$ tails~\cite{balint_komalovics}. More precisely, Chernov \& Zhang~\cite{chernovzhang2005nonlin} showed that the map is modelled by a Young tower with $O(n^{-3} (\log n)^3)$ tails and B\'alint \& Kom\'alovics~\cite{balint_komalovics} recently removed the logarithmic factor.
	\section{Preliminaries}\label{sec:prelims}
	\subsection{Probabilistic preliminaries}
	Throughout this subsection, we fix an integer $d\ge 1$. We will need the following basic interpolation inequality for the Wasserstein metric:
	\begin{lemma}\label{lem:wasserstein_interpolation}
	Let $1\le r< s$. Let $X_1,X_2$ be random elements of $D=D([0,1],\R^d)$ with laws in $ \mathcal{P}_r^s(D)$. Let $\theta = \frac{1/r-1/s}{1-1/s}$. Then 
	\[ \wass_r(X_1,X_2) \le \wass_1(X_1,X_2)^{\theta}(\wass_s(X_1,0)+\wass_s(X_2,0))^{1-\theta}. \]
\end{lemma}
\begin{proof}
	For any $\eps>0$ we can choose random elements $\tilde X_1, \tilde X_2 \in D$ defined on the same probability space such that $\tilde X_i \eqd X_i $ and $\big||\tilde X_1-\tilde X_2|_\infty\big|_1\le \wass_1(X_1,X_2)+\eps$. Note that $\big||\tilde X_i|_\infty\big|_s=\wass_s(X_i,0)$. Now by the classical $L^p$ interpolation inequality, $\pnorm{u}_r \le \pnorm{u}_1^\theta \pnorm{u}_s^{1-\theta}$ for all $u\in L^s$. It follows that
	\begin{align*}
		\pnorm{|\tilde X_1-\tilde X_2|_\infty}_r \le (\wass_1(X_1,X_2)+\eps)^\theta (\wass_s(X_1,0)+\wass_s(X_2,0))^{1-\theta}.
	\end{align*}
	Since $\eps>0$ is arbitrary and $\wass_r(X_1,X_2)\le\big||\tilde X_1-\tilde X_2|_\infty\big|_r$, the result follows.
\end{proof}
We now recall some results on sequences of independent random vectors that are useful in the proof of our main theorem.
	\begin{lemma}\label{lemma:maximal_rosenthal}
		Let $p\ge 2$. There exists a constant $C>0$ such that for all $k\ge 1$, for all independent, mean zero $\R^d$-valued random vectors $\widehat{X}_1,\dots,\widehat{X}_{k}\in L^p$,
		\begin{equation*}
			\Ex{\biggl|\max_{1\le j\le k}\Big|\sum_{i=1}^{j} \widehat{X}_i\Big| \biggr|^p}\le C\left(\biggl(\sum_{i=1}^{k} \Ex{|\widehat{X}_i|^2}\biggr)^{p/2}+ \sum_{i=1}^{k} \Ex{|\widehat{X}_i|^p}\right).
		\end{equation*}  
	\end{lemma}
	\begin{proof}
		By working componentwise, without loss of generality we may assume that $ d = 1 $. Define $ M(t)=\sum_{i=1}^{[kt]}\shat{X}_i $. Then $ M $ is a martingale. Hence by Doob's maximal inequality, \[ \pnorm{\max_{1\le j\le k}\bigg|\sum_{i=1}^{j} \widehat{X}_i\bigg|}_p=\pnorm{\sup_{0\le t\le 1}|M(t)|}_p\le \frac{p}{p-1}\pnorm{\sum_{i=1}^k \widehat{X}_i}_p. \]
		The result follows by Rosenthal's inequality~\cite{rosenthal1970subspaces}.
	\end{proof}
	We will also need the following strong invariance principle for iid random vectors, which is due to G\"otze and Za\u{\i}tsev~\cite[Theorem 3]{gotze_zaitsev_multidim_invariance}.
	\begin{lemma}\label{lemma:gotze_zaitsev}Let $ p\in(2,\infty) $. Then there exists a constant $ C>0 $ such that the following holds. Let $ (\shat X_i)_{i=1}^k $ be iid mean zero $\R^d$-valued random vectors in $ L^p $. Then there exists a probability space supporting random vectors $ (\widetilde X_i)_{i=1}^k $ with the same joint distribution as $ (\widehat X_i)_{i=1}^k $ and iid Gaussian random vectors $ (Z_i)_{i=1}^k $ with mean zero and covariance $V = k{\mathbb{E}[\widehat X_1 \otimes \widehat X_1]}$ such that
		\[ \pnorm{\max_{1\le j \le k}\bigg|\sum_{i=1}^j \widetilde X_i-\sum_{i=1}^j Z_i  \bigg|}_p\le Ck^{1/p}\big|\shat X_1\big|_p\frac{\sigma_{\max}(V)}{\sigma_{\min}(V)}, \]
		where $ \sigma^2_{\min}(V) $ and $ \sigma^2_{\max}(V) $ are the minimal and maximal positive eigenvalues of $V$, respectively. \qed
	\end{lemma}
	
	We now use this lemma to prove a weak invariance principle with rates for iid random vectors. Define $ Y_k \in D([0,1],\R^d) $ by $ Y_k(t)=\sum_{i=0}^{[kt]}\shat X_i $ for $ t\in [0,1] $.
	\begin{corol}\label{corol:wip_rates_iid}
		Let $ p\in(2,\infty) $. Then there exists $ C>0 $ such that for all $ k\ge 1 $, for any sequence of iid mean zero $\R^d$-valued random vectors $ (\shat X_i)_{i=1}^k $ in $ L^p $ we have 
		\[ \wass_p(Y_k, W)\le Ck^{1/p}\Big(\big|\shat X_1\big|_p\frac{\sigma_{\max}(V)}{\sigma_{\min}(V)}+|\Ex{\shat X_1 \otimes \shat X_1}^{1/2}|\Big), \]
		where $ W $ is a Brownian motion with covariance $ V=k\mathbb{E}[\shat X_1 \otimes \shat X_1] $.
	\end{corol}
	\begin{proof}
		Define $ A_k,\, \widetilde Y_k \in D([0,1],\R^d) $ by $ A_k(t) = \sum_{i=1}^{[kt]}Z_i $ and $ \widetilde Y_k(t)=\sum_{i=1}^{[kt]}\widetilde X_i $. Then by Lemma~\ref{lemma:gotze_zaitsev},
		\[ \pnorm{\sup_{0\le t\le 1}|\widetilde Y_k(t)-A_k(t)|}_p=\pnorm{\max_{1\le j \le k}\bigg|\sum_{i=1}^j \widetilde X_i-\sum_{i=1}^j Z_i  \bigg|}_p\le Ck^{1/p}\big|\shat X_1\big|_p\frac{\sigma_{\max}(V)}{\sigma_{\min}(V)}. \]
		It follows that 
		\begin{equation}\label{eq:wass_Y_k-A_k_bd}
			\wass_p(Y_k,A_k)=\wass_p(\widetilde Y_k,A_k)\le Ck^{1/p}\big|\shat X_1\big|_p \, \sigma_{\max}(V)/\sigma_{\min}(V).
		\end{equation} 
		Let $ W $ be a Brownian motion with covariance $ V$. Then 
		\( (W(\frac{j}{k}) - W(\frac{j-1}{k}))_{1 \le j \le k} \eqd (Z_j)_{1 \le j \le k}.\)
		Define $ \tilde A_k\in D([0,1],\R^d) $ by 
		\[ \tilde A_k(t)=\sum_{i=1}^{[kt]}\Big(W\big(\tfrac{j}{k}\big)-W\big(\tfrac{j-1}{k}\big)\Big)=W\Big(\tfrac{[kt]}{k}\Big). \] 
		It follows that $ \tilde A_k \eqd A_k $. Write $ W=V^{1/2}B $, where $ B $ is a standard Brownian motion. Then 
		\begin{align*}
			|\tilde A(t)-W(t)|&=\left|V^{1/2}\left(B(t)-B\Big(\tfrac{[kt]}{k}\Big)\right)\right|\le k^{1/2}|\Ex{\shat X_1\otimes \shat X_1}^{1/2}|\left|B(t)-B\Big(\tfrac{[kt]}{k}\Big)\right|.
		\end{align*}
		Now by~\cite[Theorem 2.1 and p.\ 28]{revuz_yor}, for all $\alpha\in (0,1/2)$,
		\[ \pnorm{\sup_{0\le s<t\le 1}\frac{|B(t)-B(s)|}{|t-s|^{\alpha}}}_p<\infty. \]
		Since $ |t-\frac{1}{k}[kt]|\le \frac{1}{k} $, it follows that
		\[ \Bigg|\sup_{0\le t\le 1}\Big|B(t)-B\Big(\tfrac{[kt]}{k}\Big)\Big|\Bigg|_p\le k^{-(\frac{1}{2}-\frac{1}{p})}\bigg|\sup_{0\le s<t\le 1}\frac{|B(t)-B(s)|}{|t-s|^{\frac{1}{2}-\frac{1}{p}}}\bigg|_p \le Ck^{-(\frac{1}{2}-\frac{1}{p})}. \]
		Hence 
		\[ \wass_p(A_k, W)=\wass_p(\tilde A_k, W)\le Ck^{1/p}|\Ex{\shat X_1\otimes \shat X_1}^{1/2}|. \]
		The result follows by combining this bound with~\eqref{eq:wass_Y_k-A_k_bd}.
	\end{proof}
	\subsection{Consequences of the Functional Correlation Bound}
	Throughout this subsection, we fix $\gamma>1$ and
	assume that $T:M\to M$ satisfies the Functional Correlation Bound with rate $n^{-\gamma}$. We first state a weak dependence lemma; in the proof of our main theorem we will use this lemma rather than applying the Functional Correlation Bound directly.
	
	Let $ e,q\ge 1 $ be integers. For $ G=(G_1,\dots,G_e):M^q\to \R^e $ and $ 0\le i<q $ we define $ [G]_{\eta,i}=\sum_{j=1}^e [G_j]_{\eta,i} $. We write $ G\in \sepdholsp{q}(M,\R^e) $ if $ \norminf{G}+\sum_{i=0}^{q-1}[G]_{\eta,i}<\infty $.
	
	Let $k\ge 1$ and consider $ k $ disjoint blocks of integers $\{\ell_i,\ell_i+1,\dots,u_i\}$, $1\le i\le k$ with $ \ell_i\le u_i<\ell_{i+1}. $ Consider $ \R^e $-valued random vectors $ X_i $ on $ (\ms,\mu) $ of the form
	\begin{equation*}
		X_i(x)=\varPhi_i(\tf^{\ell_i}x,\dots,\tf^{u_i}x)
	\end{equation*}
	where $ \varPhi_i\in \sepdholsp{u_i-\ell_i+1}(\ms,\R^e) $, $ 1\le i\le k. $
	
	When the gaps $\ell_{i+1}-u_i$ between blocks are large, the random vectors $ X_1,\dots,X_k $ are weakly dependent. Let $\widehat{X}_1,\dots,\widehat{X}_k$ be independent random vectors with $\widehat{X}_i{\eqd X_i}$.
	\begin{lemma}\label{lemma:weak_dependence}
		Let $R=\max_i \norminf{\varPhi_i}$. Then for all Lipschitz $ F:B(0,R)^k\to \R $,
		\begin{multline*}
			\big|\Ex[\mu]{F(X_1,\dots,X_{k})}-\Ex{F(\shat{X}_1,\dots,\shat{X}_{k})}\big|\\
			\le C\sum_{r=0}^{k-2}(\ell_{r+1}-u_r)^{-\gamma}\biggl(\norminf{F}+\Lip(F)\sum_{i=1}^{k}\sum_{j=0}^{u_i-\ell_i}\sdsemi{\varPhi_i}{j}\biggr),
		\end{multline*}
		where $C>0$ is the constant that appears in the Functional Correlation Bound~\eqref{eq:fcb_bd}.
	\end{lemma}
	\begin{proof}
		This is proved for $e = 1$ in \cite[Lemma 4.1]{fleming2022}. The proof of the general case
		follows the same lines.
	\end{proof}
	Let $v\in \holsp(M,\R^d)$ be mean zero. Recall that we denote $S_n = \sum_{i=0}^{n-1}v\circ T^i$.
	\begin{lemma}\label{lem:moment_bds_arbitrary_set}
		There is a constant $C>0$ such that $\pnorm{\sum_{i\in I}v\circ T^i}_{2\gamma}\le C |I|^{1/2}$ for all finite sets $I\subset \Z_{\ge 0}$.
	\end{lemma}
	\begin{proof}In the proof of~\cite[Theorem 2.4(a)]{fleming2022} an inductive argument is used to show that $\pnorm{\sum_{i=0}^{n-1}v\circ T^i}_{2\gamma}\le Cn^{1/2}$ for all $n\ge 1$.\footnote{In fact,~\cite{fleming2022} shows that a functional correlation bound for separately \textit{dynamically} H\"older functions implies analogous moment bounds for dynamically H\"older observables. It is easily checked that the same arguments apply with separately H\"older functions in place of separately dynamically H\"older ones; indeed, this is carried out in Section 3.5 of~\cite{my_thesis}.} The argument used can be adapted to prove this lemma by induction on $|I|$. Indeed, write $I = \{k_0,\dots,k_{n-1}\}$ where $k_0<\dots<k_{n-1}$ and replace the definition of $S_v(a,b)$ given in~\cite[Section 5]{fleming2022} by \( S_v(a,b) = \sum_{i=a}^{b-1} v\circ T^{k_i} \) for $0\le a<b\le n$. Then the statement of~\cite[Lemma 5.4]{fleming2022} remains valid since $n\mapsto n^{-\gamma}$ is monotonic and the gaps between the blocks $\{k_i:i\in \Z\cap[a_{2i},a_{2i+1})\}$ are no smaller than the gaps between the blocks $\{i: i \in \Z\cap[a_{2i},a_{2i+1})\}$. The remainder of the proof of~\cite[Theorem 2.4(a)]{fleming2022} goes through essentially unchanged.
	\end{proof}
	\begin{prop}\label{prop:maximal_ineq_birkhoff}
		There is a constant $ C>0 $ such that 
		$ \pnorm{\max_{1\le j\le n}|S_j|}_{2\gamma}\le Cn^{1/2}$ for all $n\ge 1$.
	\end{prop}
	\begin{proof}
		By Lemma~\ref{lem:moment_bds_arbitrary_set}, $\pnorm{S_{a+n}-S_a}_{2\gamma} = \pnorm{\sum_{i=a}^{n+a-1}v\circ T^i}_{2\gamma}\le Cn^{1/2}$ for all $n\ge 1$ and $a\ge 0$.
		The result follows by~\cite[Corollary B.2]{serfling}.
	\end{proof}
	\begin{lemma}\label{lem:var_S_n_error}The limit $\Sigma = \lim_{n\to \infty}n^{-1}\Ex[\mu]{S_n\otimes S_n}$ exists. Moreover, there is a constant $C>0$ such that $|\Ex[\mu]{S_n\otimes S_n} - n\Sigma|\le C(n^{2-\gamma}+\log n)$ for all $n\ge 1$. 
	\end{lemma}
	\begin{proof}
		We proceed by applying results from~\cite{hella_stenlund} so we need to check assumptions (SA1-SA4) of that paper. Consider the random dynamical system defined from $T:M\to M$ by taking the base space $\Omega$ to be a singleton. Then assumptions (SA2-4) are satisfied with the trivial bound $0$ for all $n$ since $T:M\to M$ is measure-preserving and ergodic. By Remark~\ref{rk:fcb_implies_doc}, (SA1) is satisfied with $\eta(n)=Cn^{-\gamma}$. The result follows by Theorem 4.1 and Lemma 4.4 from~\cite{hella_stenlund}.
	\end{proof}
	\section{Proof of the main result}\label{sec:proof_main_result}
	In this section, we prove Theorem~\ref{thm:rates_wip_lip}. The proof is based on Bernstein's classical `big block-small block' method. Let $ 0<b<a<1 $. We split $ \{0,\dots,n-1\} $ into alternating big blocks of length $ p=[n^a] $ and small blocks of length $ q=[n^b] $. Let $ k $ denote the number of big blocks, which is equal to the number of small blocks so $k = \big[\frac{n}{p+q}\big]$. The last remaining block has length at most $p+q$.
	
	Recall that $T:M\to M$ is a map satisfying the Functional Correlation Bound with rate $n^{-\gamma}$, $\gamma>1$, and that $v\in \holsp(M,\R^d)$ with mean zero.
	For $ n\ge 1 $, recall that we denote $ S_n=\sum_{r=0}^{n-1}v\circ T^r $. Let $ 1\le i\le k $ and $t\in [0,1]$. We define 
	\[ X_i = n^{-1/2}S_p\circ T^{(i-1)(p+q)}\quad \text{ and }\quad\stilde W_n(t)=\sum_{i=1}^{[kt]}X_i. \]
	Hence $ X_i $ is the sum of $ n^{-1/2}v\circ T^r $ over the $ i $th big block. We also define 
	\[ Y_i = n^{-1/2}S_q\circ T^{p+(i-1)(p+q)}\quad \text{ and }\quad R_n(t)=\sum_{i=1}^{[kt]}Y_i. \]
	Thus $ Y_i $ is the sum of $ n^{-1/2}v\circ T^r $ over the $ i $th small block.
    \begin{remark}\label{remark:chernov_clt}
           In~\cite[Section 7.8]{chernov_markarian} the `big block-small block' method is used to prove the central limit under a hypothesis on decay of multiple correlations, where the Functional Correlation Bound~\eqref{eq:fcb_bd} is
only assumed for functions $G : M^q \to \R$ of the form $G(x_0, \dots,x_{q-1}) = \prod_{i=0}^{q-1}v_i(x_i)$. Such a hypothesis is strong enough to control the characteristic function of $\widetilde{W}_n(1)$ and hence prove the central limit theorem. Assuming the Functional Correlation Bound enables us to control $\Ex[\mu]{F(\stilde{W}_n)}$ for any Lipschitz function $F:D([0,1],\R^d)\to \R$ (see the proof of Lemma~\ref{lemma:big_block_indep_approx}).     
    \end{remark}
\begin{lemma}\label{lemma:blocking_step_error}There is a constant $ C>0 $ such that 
		$ \wass_{2\gamma}(W_n, \stilde W_n+R_n)\le Cn^{(1-a)\frac{1-\gamma}{2\gamma}} $ for all $ n\ge 1 $.
	\end{lemma}
	\begin{proof}
		Let $ t\in [0,1] $. Since $ \{0,\dots,[kt](p+q)-1\} $ is the union of first $ [kt] $ big blocks and the first $ [kt] $ small blocks, 
		\[ n^{-1/2}S_{[kt](p+q)}=\sum_{i=1}^{[kt]}(X_i+Y_i)=\stilde W_n(t)+R_n(t). \]
		Now since $ n\ge k(p+q) $, we have $ [nt]\ge [kt(p+q)]\ge [kt](p+q) $. Let $ h(t)=[nt]-[kt](p+q) $. Then
		\begin{align*}
			W_n(t)&=n^{-1/2}S_{[nt]} = n^{-1/2}S_{[kt](p+q)} + n^{-1/2}S_{h(t)}\circ T^{[kt](p+q)}\\
			&=\stilde W_n(t)+R_n(t)+n^{-1/2}S_{h(t)}\circ T^{[kt](p+q)}.
		\end{align*}
		Recall that $k=\big[\frac{n}{p+q}\big] $ so $ n-k(p+q)\le p+q $. Hence 
		\begin{align*}
			h(t)\le nt - (kt - 1)(p+q)= t(n - k(p+q)) + p + q\le 2(p+q).
		\end{align*}
		It follows that 
		\begin{align*}
			\Big|W_n(t)-\big(\stilde W_n(t)+R_n(t)\big)\Big|= n^{-1/2}|S_{h(t)}|\circ T^{[kt](p+q)}\le A \circ T^{[kt](p+q)},
		\end{align*}
		where $ A=n^{-1/2}\max_{j\le 2(p+q)}|S_j| $. Hence
		\begin{align*}
			\pnorm{\sup_{0\le t\le 1}\Big|W_n(t)-\big(\stilde W_n(t)+R_n(t)\big)\Big|}^{2\gamma}_{2\gamma}&\le \Ex[\mu]{\left|\sup_{0\le t\le 1}A\circ T^{[kt](p+q)}\right|^{2\gamma}}\\
			&=\Ex[\mu]{\left|\max_{0\le i\le k}A\circ T^{i(p+q)}\right|^{2\gamma}}.
		\end{align*}
		Now by Proposition~\ref{prop:maximal_ineq_birkhoff}, we have $ \pnorm{A}_{2\gamma}\ll n^{-1/2}(p+q)^{1/2}\ll n^{(a-1)/2} $. It follows that
		\begin{align*}
			\Ex[\mu]{\left|\max_{0\le i\le k}A\circ T^{i(p+q)}\right|^{2\gamma}}&\ll \Ex[\mu]{\sum_{i=0}^{k}|A\circ T^{i(p+q)}|^{2\gamma}}=(k+1)\Ex[\mu]{|A|^{2\gamma}}\\
			&\ll n^{1-a}n^{\gamma(a-1)}=n^{(1-a)(1-\gamma)},
		\end{align*}
		which completes the proof.
	\end{proof}
Next we control the contribution of the small blocks.
\begin{lemma}\label{lem:small_blocks_negligible}
	There is a constant $C>0$ such that $\wass_{2\gamma}(\stilde W_n+R_n, \stilde W_n)\le Cn^{(b-a)/2}$ for all $n\ge 1$.
\end{lemma}
\begin{proof}
	We proceed by bounding $\pnorm{\sup_{0\le t\le 1}|R_n(t)|}_{2\gamma}$. By working componentwise, without loss of generality we may assume that $ d = 1 $. For $m\ge 1$, define $A_m = \sum_{i=1}^m Y_i$. Then $A_m$ is the sum of $n^{-1/2} v\circ T^i$ over $m$ disjoint blocks of length $q$ so by Lemma~\ref{lem:moment_bds_arbitrary_set}, $\pnorm{A_m}_{2\gamma}\ll n^{-1/2}m^{1/2}q^{1/2}$. Since $Y_i$ is a stationary sequence of mean zero random variables in $L^{2\gamma}$, by~\cite[Proposition 1(ii)]{wu_strong_invariance}, we have that
	\begin{align*}
		\pnorm{\max_{1\le m \le 2^N}|A_m|}_{2\gamma}\le \sum_{j=0}^N 2^{\frac{N-j}{2\gamma}}\pnorm{A_{2^j}}_{2\gamma}.
	\end{align*}
	for all $N\ge 1$. Thus 
	\begin{align*}
		\pnorm{\max_{1\le m \le 2^N}|A_m|}_{2\gamma}
		&\ll n^{-1/2}q^{1/2}\sum_{j=0}^N 2^{\frac{N-j}{2\gamma}}2^{j/2}
		\ll n^{-1/2}q^{1/2}2^{N/2}.
	\end{align*}
	By taking $N = \lceil \log_2(k)\rceil$, it follows that 
	\[ \pnorm{\sup_{0\le t\le 1}|R_n(t)|}_{2\gamma}=\pnorm{\max_{1\le m \le k}|A_m|}_{2\gamma}\ll n^{-1/2}q^{1/2}k^{1/2}\ll n^{(b-a)/2}.\qedhere \]
\end{proof}
	Let $ \shat X_1,\dots,\shat X_k $ be independent random vectors such that $ \shat{X}_i \eqd X_i $, and define $ \shat W_n(t)=\sum_{i=1}^{[kt]}\shat X_i $ for $ t\in [0,1] $.
	\begin{lemma}\label{lemma:big_block_indep_approx}
		There is a constant $C>0$ such that $\kant(\stilde W_n, \shat W_n)\le Cn^{3/2-a-b\gamma}$ for all $n\ge 1$.
	\end{lemma}
\begin{proof}
	By the Kantorovich-Rubinstein Theorem~\eqref{eq:kantorovich_rubenstein}, it suffices to show that there is a constant $C>0$ such that \[|\Ex[\mu]{G(\stilde W_n)} - \Ex[\mu]{G(\shat W_n)}|\le Cn^{3/2-a-b\gamma}\Lip(G)\] for any Lipschitz function $G:{D([0,1],\R^d)\to \R}$ and $n\ge 1$.

		We proceed by using Lemma~\ref{lemma:weak_dependence}. First note that
		$ X_i(x)=\varPhi_i(T^{\ell_i}x,\dots,T^{u_i}x) $
		where $ \ell_i=(i-1)(p+q)$, $u_i=\ell_i+p-1 $ and \[ \varPhi_i(y_0,\dots,y_{u_i-\ell_i})=n^{-1/2}\sum_{r=0}^{u_i-\ell_i}v(y_r). \]
		Let $ 0\le r\le u_i-\ell_i $. Then $ \sdsemi{\varPhi_i}{r}\le n^{-1/2}\dsemi{v_n}.$ Let $ R=\max_i \norminf{\varPhi_i}\le pn^{-1/2}\norminf{v_n}.$ Define $ \pi_k:B(0,R)^k\to D([0,1],\R^d) $ by
		\[ \pi_k(x_1,\dots,x_k)(t)=\sum_{i=1}^{[kt]}x_i. \]
		for $ t\in [0,1] $. Then $ \stilde W_n=\pi_k(X_1,\dots,X_k) $ and $ \shat W_n = \pi_k(\shat X_1,\dots,\shat X_k) $.
		
		Now for all $ (x_1,\dots,x_k),(x'_1,\dots,x'_k)\in B(0,R)^k $,
		\[ \sup_{t\in[0,1]}\Bigg|\sum_{i=1}^{[kt]}x_i-\sum_{i=1}^{[kt]}x'_i\Bigg|\le \sum_{i=1}^k|x_i-x_i'|, \]
		so $ \Lip(\pi_k)\le 1 $. Without loss of generality we may assume that $ G(0)=0 $, for otherwise we can consider $ G'=G-G(0) $. Thus 
		\begin{equation}\label{eq:G_circ_pi_norminf}
			\norminf{G\circ \pi_k}\le \Lip(G)\norminf{\pi_k}\le kR\Lip(G)\ll n^{1/2}\Lip(G).
		\end{equation}
		By applying Lemma~\ref{lemma:weak_dependence} with $ F=G\circ \pi_k $, we obtain that 
		$ |\Ex[\mu_n]{G(\stilde W_n)}-\Ex{G(\shat W_n)}|\le A,$ where 
		\begin{align*}
			A &= C\sum_{r=1}^{k-1}(\ell_{r+1}-u_r)^{-\gamma}\biggl(\norminf{G\circ \pi_k}+\Lip(G\circ \pi_k)\sum_{i=1}^{k}\sum_{j=0}^{u_i-\ell_i}\sdsemi{\varPhi_i}{j}\biggr).
		\end{align*}
		Note that $ u_i-\ell_i=p-1 $ and $ \ell_{i+1}-u_i=q+1\ge n^b $ for $ 1\le i\le k $. Thus
		\begin{equation*}
			\sum_{r=1}^{k-1}(\ell_{r+1}-u_r)^{-\gamma}\le kn^{-b\gamma}\ll n^{1-a-b\gamma}
		\end{equation*}
		and 
		\begin{equation*}
			\sum_{i=1}^{k}\sum_{j=0}^{u_i-\ell_i}\sdsemi{\varPhi_i}{j}\le kpn^{-1/2}\dsemi{v_n}\le n^{1/2}\dsemi{v_n}.
		\end{equation*}
		By combining these bounds with~\eqref{eq:G_circ_pi_norminf} it follows that
		\begin{align*}
			A&\ll n^{1-a-b\gamma}\biggl(\norminf{G\circ \pi_k}+\Lip(G\circ \pi_k)n^{1/2}\biggr)\\
			&\ll n^{1-a-b\gamma}\biggl(\Lip(G)\,n^{1/2}+\Lip(G)\Lip(\pi_k)n^{1/2}\biggr)\le 2n^{3/2-a-b\gamma}\Lip(G).\qedhere
		\end{align*}
	\end{proof}
	Let $ V_n = k\Ex{\shat X_1 \otimes \shat X_1} $, so $ V_n = \Ex{\shat W_n(1)\otimes \shat W_n(1)} $ is the covariance of $ \shat W_n(1) $. Before proceeding further, we need a bound on the difference between $ V_n $ and the limiting covariance $ \Sigma $.
	\begin{lemma}\label{lemma:indep_approx_covar_error}
		There is a constant $ C>0 $ such that for all $ n\ge 1 $, \[ |V_n - \Sigma|\le C(n^{a(1-\gamma)}+n^{b-a}+n^{a-1}). \]
	\end{lemma}
	\begin{proof}
		Since $\shat X_1\eqd X_1 $, we have
		$V_n = k\Ex[\mu]{ X_1 \otimes X_1}=kn^{-1}\Ex[\mu]{S_p\otimes S_p}.$ By Lemma~\ref{lem:var_S_n_error},
		$ |p^{-1}\Ex{S_p\otimes S_p} - \Sigma|\ll p^{-1}\log p+p^{1-\gamma}. $
		Since $ kp\le n $ and $ p=[n^a]$, it follows that
		\begin{equation*}
			\big|V_n - \tfrac{kp}{n}\Sigma\big|=\tfrac{kp}{n}\big|p^{-1}\Ex{S_p\otimes S_p} - \Sigma\big|\ll n^{-a}\log n+n^{a(1-\gamma)}.
		\end{equation*}
		Now $ n-k(p+q)\le p + q $ so 
		\[ \big|\big(\tfrac{kp}{n}-1\big)\Sigma\big|\le \tfrac{1}{n}(kq+p+q)|\Sigma|\ll \tfrac{1}{n}(n^{1-a}n^b+n^a)=n^{b-a}+n^{a-1}. \]
		Hence 
		\begin{equation*}
			|V_n - \Sigma| \ll n^{-a}\log n+n^{a(1-\gamma)}+n^{b-a}+n^{a-1}\ll n^{a(1-\gamma)}+n^{b-a}+n^{a-1},
		\end{equation*}
		as required.
	\end{proof}
	\begin{lemma}\label{lemma:approx_indep_big_blocks_bm}
		There is a constant $ C>0 $ such that for all $ n\ge 1 $,
		\[ \wass_{2\gamma}(\shat W_n,W) \le C(n^{a(1-\gamma)/2}+ n^{(b-a)/2}+n^{(1-a)\frac{1-\gamma}{2\gamma}}). \]
	\end{lemma}
\begin{proof}
	We first consider the special cases where the covariance matrix $\Sigma$ is positive definite and $\Sigma=0$.
	
	Let $\Sigma$ be positive definite. By Corollary~\ref{corol:wip_rates_iid}, 
	\begin{equation*}
		\wass_{2\gamma}(\shat W_n, W^{(V_n)})\le Ck^{1/(2\gamma)}\Big(\big|\shat X_1\big|_{2\gamma}\frac{\sigma_{\max}(V_n)}{\sigma_{\min}(V_n)}+|\Ex{\shat X_1 \otimes \shat X_1}^{1/2}|\Big),
	\end{equation*}
	where $ W^{(V_n)} $ is a Brownian motion with covariance $ V_n $. By Lemma~\ref{lemma:indep_approx_covar_error},  $ V_n \to \Sigma $. Since $ \Sigma $ is positive definite, it follows that $ \sigma_{\max}(V_n)/\sigma_{\min}(V_n)\to \sigma_{\max}(\Sigma)/\sigma_{\min}(\Sigma) $. Moreover, $ V_n^{1/2}\to \Sigma^{1/2} $ so 
	\[ |\Ex{\shat X_1 \otimes \shat X_1}^{1/2}|=k^{-1/2}|V_n^{1/2}|\ll k^{-1/2}. \]
	Now $ \shat X_1 \eqd X_1 = n^{-1/2}S_p $ so by Proposition~\ref{prop:maximal_ineq_birkhoff}, we have $ \pnorm{\shat X_1}_{2\gamma}\ll n^{-1/2}p^{1/2}$. 
	Since $ k=\big[\frac{n}{p+q}\big]\sim n/p\sim n^{1-a} $, it follows that 
	\begin{align}\label{eq:pnorm_sum_bd}
		k^{1/(2\gamma)}\pnorm{\shat X_1}_{2\gamma}\ll n^{(1-a)/(2\gamma)}n^{-1/2+a/2}=n^{(1-a)\frac{1-\gamma}{2\gamma}}
	\end{align}
	and hence that
	\begin{align}
		\wass_{2\gamma}(\shat W_n, W^{(V_n)})&\ll n^{(1-a)\frac{1-\gamma}{2\gamma}}+ k^{\frac{1}{2\gamma}-\frac{1}{2}}\ll n^{(1-a)\frac{1-\gamma}{2\gamma}}.\label{eq:approx_brownian_V_n}
	\end{align}
	Write $ W^{(V_n)}=V_n^{1/2}B $, where $ B $ is a standard Brownian motion. Let $ W=\Sigma^{1/2}B $, so $ W $ is a Brownian motion with covariance $ \Sigma $. Note that
	\[ |W^{(V_n)}(t)-W(t)|=|(V^{1/2}_n - \Sigma^{1/2})B(t)|\le |V^{1/2}_n - \Sigma^{1/2}||B(t)|. \]
	Now since $ \Sigma $ is positive definite, by~\cite[equation (7.2.13)]{horn_johnson}, 
	\[ \|V_n^{1/2}-\Sigma^{1/2}\|_2\le \|\Sigma^{-1/2}\|_2\|V_n - \Sigma\|_2,  \]
	where $ \|\cdot\|_2 $ denotes the spectral norm. Hence by Lemma~\ref{lemma:indep_approx_covar_error},
	\[ \pnorm{\sup_{0\le t\le 1}|W^{(V_n)}(t)-W(t)|}_{2\gamma}\ll |V_n - \Sigma|\pnorm{\sup_{0\le t\le 1}|B(t)|}_{2\gamma}\ll n^{a(1-\gamma)}+n^{b-a}+n^{a-1}. \]
	By combining this bound with~\eqref{eq:approx_brownian_V_n}, it follows that 
	\begin{align}
		\wass_{2\gamma}(\shat W_n, W)&\le \wass_{2\gamma}(\shat W_n, W^{(V_n)})+\wass_{2\gamma}(W^{(V_n)}, W)\nonumber\\
		&\ll n^{(1-a)\frac{1-\gamma}{2\gamma}}+n^{a(1-\gamma)}+n^{b-a}+n^{a-1}\nonumber\\
		&\ll n^{(1-a)\frac{1-\gamma}{2\gamma}}+n^{a(1-\gamma)}+n^{b-a}.\label{eq:bound_nonsingular_case}
	\end{align}
	\par Let $\Sigma=0$. Then by Lemma~\ref{lemma:indep_approx_covar_error}, \[k\Ex{|\shat{X}_1|^2}\le k|\Ex{\shat{X}_1\otimes \shat{X}_1}|\ll n^{a(1-\gamma)}+n^{b-a}+n^{a-1}.\]
	Hence by Lemma~\ref{lemma:maximal_rosenthal} and~\eqref{eq:pnorm_sum_bd},
	\begin{align}\label{eq:bound_0_covar_case}
		\wass_{2\gamma}(\shat{W}_n,0)&=\pnorm{\max_{1\le j\le k}\Big|\sum_{i=1}^{j} \widehat{X}_i\Big|}_{2\gamma}\ll \biggl(\sum_{i=1}^{k} \Ex{|\shat{X}_i|^2}\biggr)^{\frac{1}{2}}+ \biggl(\sum_{i=1}^{k} \Ex{|\shat{X}_i|^{2\gamma}}\biggr)^{\frac{1}{2\gamma}} \nonumber\\
		&= (k\Ex{|\shat{X}_1|^2})^{\frac{1}{2}}+k^{\frac{1}{2\gamma}}\pnorm{X_1}_{2\gamma}
		\ll n^{a(1-\gamma)/2} + n^{(b-a)/2} + n^{(a-1)/2}+n^{(1-a)\frac{1-\gamma}{2\gamma}}\nonumber\\
		&\ll n^{a(1-\gamma)/2} + n^{(b-a)/2} +n^{(1-a)\frac{1-\gamma}{2\gamma}}.
	\end{align}
	
	Finally, we consider the case where $\Sigma$ is an arbitrary symmetric positive semidefinite matrix. 
	Then we can write $\R^d$ as a direct sum $\R^d = E\oplus F$ such that $\Sigma$ is nonsingular on $E$ and vanishes along $F$. Let $\pi_E$ denote the projection onto $E$. Then $W = \pi_E W$ so combining~\eqref{eq:bound_0_covar_case} and~\eqref{eq:bound_nonsingular_case} yields that
	\begin{align*}
		\wass_{2\gamma}(\shat{W}_n,W)&\le \wass_{2\gamma}(\shat{W}_n, \pi_E \shat{W}_n) + \wass_{2\gamma}(\pi_E \shat{W}_n,\pi_E W)\\
		&\le \pnorm{\sup_{0\le t\le 1}\big|\shat{W}_n(t)-\pi_E \shat{W}_n(t)\big|}_{2\gamma}+\wass_{2\gamma}(\pi_E \shat{W}_n,\pi_E W)\\
		&\ll n^{a(1-\gamma)/2} + n^{(b-a)/2} + n^{(1-a)\frac{1-\gamma}{2\gamma}}+n^{(1-a)\frac{1-\gamma}{2\gamma}}+n^{a(1-\gamma)}+n^{b-a}\\
		&\ll n^{a(1-\gamma)/2}+ n^{(b-a)/2}+n^{(1-a)\frac{1-\gamma}{2\gamma}},
	\end{align*}
	as required.
\end{proof}
	We now have all the estimates required to prove our main theorem.
\begin{proof}[Proof of Theorem~\ref{thm:rates_wip_lip}]
	By Lemmas~\ref{lemma:blocking_step_error} and~\ref{lem:small_blocks_negligible},
	\begin{align}
		\wass_{2\gamma}(W_n,\stilde W_n)&\le \wass_{2\gamma}(W_n, \stilde W_n+R_n)+\wass_{2\gamma}(\stilde W_n+R_n, \stilde W_n)\nonumber\\
		&\ll n^{(1-a)\frac{1-\gamma}{2\gamma}}+n^{(b-a)/2}.\label{eq:diff_with_big_blocks}
	\end{align}
	On the other hand, by Lemma~\ref{lemma:approx_indep_big_blocks_bm},
	\begin{align}
		\wass_{2\gamma}(\shat W_n,W)\ll n^{a(1-\gamma)/2}+ n^{(b-a)/2}+n^{(1-a)\frac{1-\gamma}{2\gamma}}.\label{eq:indep_blocks_vs_brownian}
	\end{align}
	Now by Proposition~\ref{prop:maximal_ineq_birkhoff}, $\wass_{2\gamma}(W_n,0)=O(1)$ and hence by~\eqref{eq:diff_with_big_blocks} we have that $\wass_{2\gamma}(\stilde W_n, 0)=O(1)$. Similarly,~\eqref{eq:indep_blocks_vs_brownian} implies that $\wass_{2\gamma}(\shat W_n, 0)=O(1)$. Hence by Lemma~\ref{lemma:big_block_indep_approx} and Lemma~\ref{lem:wasserstein_interpolation}, it follows that 
	\begin{align*}
		\wass_r(\stilde W_n, \shat W_n)\ll \wass_1(\stilde W_n, \shat W_n)^\frac{2\gamma-r}{r(2\gamma-1)}\ll n^{\frac{2\gamma-r}{r(2\gamma-1)}(3/2-a-b\gamma)}.
	\end{align*}
	By combining this with~\eqref{eq:diff_with_big_blocks} and~\eqref{eq:indep_blocks_vs_brownian}, we obtain that $ \wass_r(W_n, W)\ll n^{-\kappa(\gamma,r)} $ where
	\begin{align}
		\kappa(\gamma,r)=\min\Big\{(1-a)\frac{\gamma-1}{2\gamma},\, \frac{(2\gamma-r)(a+b\gamma-\frac{3}{2})}{r(2\gamma-1)},\, \frac{a-b}{2},\, \frac{a(\gamma-1)}{2}\Big\}.\label{eq:wip_implicit_Wp_bd}
	\end{align}
	Recall that $ 0<b<a<1 $ are arbitrary. Finally, we show how to choose $ a $ and $ b $ in order to obtain the expression for $ \kappa(\gamma,r) $ in the statement of this theorem. We choose $a$ and $b$ such that 
	\[ (1-a)\frac{\gamma-1}{2\gamma}=\frac{(2\gamma-r)(a+b\gamma-\frac{3}{2})}{r(2\gamma-1)}=\frac{a-b}{2}. \]
	Note that $ (1-a)\frac{\gamma-1}{2\gamma}=\frac{(2\gamma-r)(a+b\gamma-\frac{3}{2})}{r(2\gamma-1)}$ if and only if
	\begin{equation}\label{eq:b_eq_1}
		b = \frac{a(r(2\gamma+1)-4\gamma)+6\gamma-3r}{4\gamma^2-r}.
	\end{equation}
	On the other hand, $(1-a)\frac{\gamma-1}{2\gamma}=\frac{a-b}{2}$ if and only if 
	\begin{equation}\label{eq:b_eq_2}
		b = \frac{a(2\gamma-1)-\gamma+1}{\gamma}.
	\end{equation}
	By equating~\eqref{eq:b_eq_1} and~\eqref{eq:b_eq_2} and solving for $a$ we obtain that 
	\[ a=\frac{4\gamma^3+2\gamma^2 - 4r\gamma + r}{8\gamma^3 - r(2\gamma^2 + 3\gamma - 1)}\quad \text{ and }\quad (1-a)\frac{\gamma-1}{2\gamma}=\frac{(2\gamma-r)(2\gamma-1)(\gamma-1)}{2(8\gamma^3 - r(2\gamma^2 + 3\gamma - 1))}{.}
	\]
	Now since $r\le 2\gamma$, we have that $2\gamma^2-4r\gamma\ge -2\gamma^2-2r\gamma$ so
	\begin{align*}
		a \ge \frac{4\gamma^3-2\gamma^2 - 2r\gamma + r}{8\gamma^3 - r(2\gamma^2 + 3\gamma - 1)}=\frac{(2\gamma^2-r)(2\gamma-1)}{8\gamma^3 - r(2\gamma^2 + 3\gamma - 1)}
		\ge \frac{1-a}{\gamma}.
	\end{align*}
	It follows that
	\[ \kappa(\gamma,r)=\min\Big\{(1-a)\frac{\gamma-1}{2\gamma},\frac{a(\gamma-1)}{2} \Big\}=(1-a)\frac{\gamma-1}{2\gamma}=\frac{(2\gamma-r)(2\gamma-1)(\gamma-1)}{2(8\gamma^3 - r(2\gamma^2 + 3\gamma - 1))}.\]
	Since $\kappa(\gamma,r)>0$, by~\eqref{eq:wip_implicit_Wp_bd} we have that $0<b<a<1$, as required.
\end{proof}
\bibliographystyle{alpha}
\bibliography{rates_wip_functional_correlation}
\end{document}